\documentclass{elsarticle}
\usepackage{amssymb, amscd, verbatim, amsthm, amsmath, MnSymbol, stmaryrd}
\numberwithin{equation}{section}
\newtheorem{theorem}{Theorem}[section]
\newtheorem{lemma}[theorem]{Lemma}
\newtheorem{corollary}[theorem]{Corollary}
\newtheorem{prop}[theorem]{Proposition}

\theoremstyle{definition}

\def\tr{\Delta}
\def\bea{\begin{eqnarray*}}
\def\eea{\end{eqnarray*}}
\def\be{\begin{eqnarray}}
\def\ee{\end{eqnarray}}
\def\a{\alpha}

\def\n{\nabla}

\begin{document}

\begin{frontmatter}
 \title{Vacuum static spaces with vanishing of complete divergence of Bach tensor and Weyl tensor } 

\author[sh]{Seungsu Hwang}
\ead{seungsu@cau.ac.kr} 
\author[gy]{Gabjin Yun}
\ead{gabjin@mju.ac.kr}

\address[sh]{Department of Mathematics, Chung-Ang University,
84 HeukSeok-ro DongJak-gu, Seoul 06974, Republic of Korea.
}
\address[gy]{ Department of Mathematics, Myong Ji University, 
116 Myongji-ro Cheoin-gu, Yongin, Gyeonggi 17058, Republic of Korea. } 


 \begin{abstract}
 In this paper, we study vacuum static spaces with 
the complete divergence of the Bach tensor and Weyl tensor. First, we prove that the vanishing of complete divergence of the Bach tensor and Weyl tensor implies the harmonicity of the metric, and we present examples in which these conditions do not imply Bach flatness. As an application, we prove the non-existence of multiple black holes in vacuum static spaces with zero scalar curvature. On the other hand, we prove the Besse conjecture under these conditions, which are weaker than harmonicity or Bach flatness. Moreover, we show a rigidity result for vacuum static spaces and find a sufficient condition for the metric to be Bach-flat. 
 \end{abstract}

\end{frontmatter}
\section{Introduction}

An $n$-dimensional complete Riemannian manifold $(M,g)$ is said to be a {\it static space} with 
a perfect fluid if there exists a smooth non-trivial function $f$ on $M$ satisfying 
\be 
D_gdf -\left( r_g -\frac {s_g}{n-1}g\right) f =\frac 1n \left( \frac {s_g}{n-1}\,f +\tr_g f\right)g,
\label{eqn1}\ee
where $D_gdf$ is the Hessian of $f$,  $r_g$ is the Ricci tensor of $g$ with its scalar curvature 
$s_g$,  and
$\tr_g f$ is the (negative) Laplacian of $f$. In particular, if 
\be
\tr_g f =-\frac {s_g}{n-1}\, f, \label{eqn2018-9-5-1}
\ee
$(M,g)$ is said to be a {\it vacuum static space}. In this case, equation (\ref{eqn1}) reduces to 
\be 
D_gdf  - \left( r_g -\frac {s_g}{n-1}g\right) f =0. \label{static1}
\ee
The above equation was considered by
Fischer and Marsden (\cite{f-m}) in their study of the surjectivity of a linearized  scalar curvature 
functional in the space of Riemannian metrics (cf. \cite{kob}, \cite{shen}).
More precisely, the linearized scalar curvature $s_g'$ is given by
$$ 
s_g'(h) =-\tr_g {\rm tr}_g\, h +\delta\delta h - \langle r_g, h\rangle
$$
for any  symmetric bilinear form $h$ on $M$ (cf. \cite{Be}). Here, $\delta = - {\rm div}$ 
is the (negative) divergence, which is defined by 
$\displaystyle{\delta h (X) = -\sum_{i=1}^n D_{E_i}h(E_i, X)}$ for any vector $X$ and a local frame
$\{E_i\}$.
Then, the $L^2$-adjoint operator
 $s_g'^*$ with respect to the metric $g$ is given by
\be 
s_g'^*(f)=D_gd f-(\tr_g f)g - f\, r_g\label{eqn2016-9-5-4}
\ee
for any smooth function $f$ on $M$. Thus, if a smooth function $f$ on $M$ is a solution of the 
vacuum static equation (\ref{static1}), then $f$ 
 is an element of the kernel space, $\ker s_g'^*$, of  the operator $s_g'^*$.
By taking the divergence of (\ref{static1}), we have $\frac 12 f\, ds_g=0$, which implies that $s_g$ is constant on $M$ since there is no critical point of $f$ in $f^{-1}(0)$ (see, for example, \cite{Bour}). When $M$ is compact, it is  known (\cite{Bour}) that a compact vacuum static space is either isometric to a Ricci-flat manifold with $\ker s_g'^* ={\Bbb R}\cdot 1$ and $s_g=0$, or the scalar curvature $s_g$ is a strictly positive constant and $\frac{s_g}{n-1}$ is an eigenvalue of the Laplacian.
It turns out that the warped product manifold 
$(M\times_{f} {\mathbb R}, g\pm f^2 dt^2)$ is Einstein when $f \in \ker s_g'^*$ (\cite{corv}).

Some rigidity results related to vacuum static spaces have been found. For example, 
 Qing and Yuan showed (\cite{QY}) that, if $(M,g,f)$ is a Bach-flat vacuum 
static space with compact level sets of $f$,  then it is either Ricci-flat, isometric to $S^n$,  $H^n$,  
or  a warped product space. Recall that the Bach tensor (see Section 2 for definition) 
discussed first by Bach in \cite{bach}
is deeply related  to general relativity and conformal geometry (cf. \cite{lis}), and
in  dimension $n=4$, it is well known (\cite{Be}) that the Bach tensor 
is conformally invariant and arises as a gradient of the total Weyl curvature 
functional, which is given by the integral of the square norm of the Weyl tensor. On the other hand, Kim and Shin proved a local 
classification of  four-dimensional  vacuum static spaces with harmonic curvature, or ${\rm div}\, { R}=0$ (\cite{ks}).
We say that a Riemannian manifold $(M, g)$ has harmonic curvature
if ${\rm div}\,  R = 0$, or equivalently, that the Ricci tensor $r_g$ is a Codazzi tensor.

Therefore, it is interesting to investigate the rigidity of vacuum static spaces under weaker curvature 
conditions than $B=0$ or ${\rm div}\, {R}=0$. Towards this objective, Catino, Mastrolia, and Monticelli (\cite{CMM})
considered the complete divergence of the Weyl
tensor to classify gradient Ricci solitons. They proved 
a classification theorem for gradient shrinking Ricci solitons
of dimensions $n \ge 4$ under the following curvature condition: $$
\mbox{div}^4 {\mathcal W} =0.$$

The purpose of this paper is to find rigidity results for vacuum static spaces under weaker curvature 
conditions. First, we consider the vanishing of complete divergence of the Bach tensor 
with ${\rm div}^4 {\mathcal W}=0$.
It is natural to inquire whether ${\rm div}^2 B=0$ with ${\rm div}^4 {\mathcal W}=0$ on vacuum 
 static spaces guarantees Bach flatness. However, because there exist 
 $n$-dimensional vacuum static spaces that satisfy ${\rm div}^2 B=0$ 
 and ${\rm div}^4{\mathcal W}=0$ with $B\neq 0$ for $n\geq 4$ (Proposition~\ref{main1}), Bach flatness is not guaranteed. 
In this paper, we prove that  ${\rm div}^2 B=0$ with ${\rm div}^4 {\mathcal W}=0$ on an 
$n$-dimensional complete (non-compact) vacuum static space does imply the harmonicity of 
the metric for $n\geq 5$.

\begin{theorem}\label{main3}
Let $(M^n, g, f), n \ge 5,$ be an $n$-dimensional  complete vacuum static space
with compact level sets of $f$. If ${\rm div}^2 B=0$ and ${\rm div}^4 {\mathcal W}=0$, 
then $(M,g)$ has harmonic curvature. In particular, ${\rm div} B=0$.
\end{theorem}

Note that ${\rm div}^2 B=0$ is not an additional condition on a four-dimensional manifold, 
since ${\rm div} B=0$ always holds for $n=4$ (see Proposition~\ref{lem2018-9-25-1}). 
Thus, we consider a rather different  condition for $n=4$. 
By Corollary~\ref{ddivb}, ${\rm div}^2 B=0$ if and only if 
\be 
\frac{1}{2}|C|^2 = - \langle {\rm div} C, z_g\rangle,\label{dcond}
\ee 
where $C$ is the Cotton tensor and $z_g$ is the traceless Ricci tensor. Thus, for $n=4$, instead of the condition ${\rm div}^2 B=0$, (\ref{dcond}) is an appropriate condition. Recall that the Cotton tensor $C\in \Gamma(\Lambda^2M\otimes T^*M)$ is defined as
\be 
C=d^Dr_g -\frac 1{2(n-1)}\, {ds_g}\wedge g. \label{eqn03-1}
\ee
Here,  $ds_g$ denotes the usual total differential of $s_g$, and
for a $1$-form  $\phi$ and a symmetric $2$-tensor $\eta \in C^{\infty}(S^2M)$,
$\phi \wedge \eta $ is defined by
$$
(\phi \wedge \eta) (X,Y,Z)= \phi(X) \eta(Y,Z)-\phi(Y) \eta(X,Z).
$$

\begin{theorem}\label{4dim1}
Let $(M, g, f)$ be a four-dimensional  complete vacuum static space with  compact level sets of $f$. If  
 $$
 \frac 12 |C|^2 = - \langle {\rm div} C, z_g\rangle
 $$ 
 and ${\rm div}^4 {\mathcal W}=0$, then $M$ has harmonic curvature.
\end{theorem}

When the scalar curvature vanishes, by replacing $f$ with $h$, 
 (\ref{static1}) reduces to the {\it static vacuum Einstein equation}: 
\begin{eqnarray}
\begin{array}{ll}
h\, r_g = D_gdh, \\
\Delta_g h = 0.
\end{array}\label{static:eq1}
\end{eqnarray}
 Let $(M,g,h)$ be  a non-trivial solution of (\ref{static:eq1}) with $h\geq 0$, which is connected and complete up to the boundary, and
assume that  the set  $h^{-1}(0)=\partial M$ is compact and that $g$ extends smoothly to $\partial M$. 
 The set $h^{-1}(0)$ is called the horizon and is the boundary of black holes in general relativity. 
It was proved in \cite{grg} that no multiple black holes exist in static vacuum spacetime if $(M,g)$ has harmonic curvature. As a consequence of Theorem~\ref{main3} and \ref{4dim1}, we show the nonexistence of multiple black holes under ${\rm div}^4 {\mathcal W}=0$ and 
${\rm div}^2 B=0$ for $n\geq 5$ and $\frac 12 |C|^2=-\langle {\rm div} C, z_g\rangle$ in the case of $n=4$.

\begin{theorem}\label{main2}
Let $(g, h)$ be  a non-trivial solution (\ref{static:eq1}) on an  $n$-dimensional manifold $M$ with the compact set $h^{-1}(0)=\partial M$. Assume that ${\rm div}^2 B=0$ for $n\geq 5$, or  $\frac 12 |C|^2=-\langle {\rm div} \, C, z_g\rangle$ for $n=4$.
If ${\rm div}^4{\mathcal W}=0$, then no multiple black holes exist in $M$. 
\end{theorem}
As another application, we prove the Besse conjecture under these conditions. It has been conjectured in \cite{Be} that a non-trivial solution $(g,f)$ of the equation 
\be s_g'^*(f)=z_g\label{eqn22}\ee
on a compact manifold $(M,g)$ is isometric to a standard sphere. We have the following theorem.

\begin{theorem} \label{main2-100}
Let $(g, f)$ be a non-trivial solution of (\ref{eqn22}) on an $n$-dimensional compact manifold $M$. Assume that ${\rm div}^2 B=0$ for $n\geq 5$, or $ \frac 12 |C|^2=-\langle {\rm div}\,  C, z_g\rangle$ for $n=4$. If ${\rm div}^4{\mathcal W}=0$, then the Besse conjecture holds. \end{theorem}
It was proved in \cite{QY} that, when $n=3$, the Besse conjecture holds 
if ${\rm div}^2 B=0$.


On the other hand, for a vacuum static space $(M,g,f)$ satisfying ${\rm div}^2 B=0$, we prove that $i_{\nabla f}B=0$ if and only if $B=0$ (Corollary~\ref{cor122}). Moreover, we have the following rigidity result.

\begin{theorem}\label{main12}
Let $(M^n, g, f)$ be an $n$-dimensional compact vacuum static space with $i_{\n f}B = 0$.
If $|z|^2$ attains its local maximum on the set $M\setminus f^{-1}(0)$, then
 $(M, g)$ is isometric to a standard sphere. 
\end{theorem}
To prove our main results, we introduce a $3$-tensor $T$ defined as  
\be
T= \frac 1{n-2}\, df  \wedge z +\frac 1{(n-1)(n-2)}\, i_{\nabla f }z \wedge g, \label{defnt}\ee
which follows naturally from (\ref{static1}). The complete divergence of $T$ also implies Bach flatness.
\begin{theorem} \label{thm101b}
Let $n \ge 5$ and $(M, g, f)$ be a non-trivial vacuum static space of compact level sets 
of $f$ with ${\rm div}^2  B=0$ for $n\geq 5$, or $ \frac 12 |C|^2=-\langle {\rm div}\,  C, z_g\rangle$ for $n=4$.  If ${\rm div}^3 T=0$ and $f$ has its minimum (or maximum) in $M$, then $(M, g)$ is Bach-flat.
\end{theorem}

\section{Preliminaries}

The Bach tensor $B$ of an $n$-dimensional Riemannian manifold $(M, g)$ 
is defined as
\be
B=\frac 1{n-3}\,{\rm div}_1{\rm div}_4\, {\mathcal W}+ \frac 1{n-2}\, \mathring{\mathcal W}{r_g},
\label{bach}
\ee
where ${\mathcal W}$ is the Weyl tensor and
$$
\mathring{\mathcal W}{r_g}(X,Y)=\sum_{i=1}^n r_g({\mathcal W}(X, E_i)Y, E_i)$$
for an orthonormal frame $\{E_i\}_{i=1}^n$.  
A Riemannian metric is said to be Bach-flat if the Bach tensor $B$ of the metric vanishes. 
It is easy to see that a locally and conformally flat
 metric has
a vanishing Bach tensor. For $n=4$ dimensions, it is also well-known that metrics 
that  are locally conformal to an Einstein metric as well as half-conformally flat metrics have vanishing Bach tensors.

On the other hand, the Bach tensor $B$ is related to the Cotton tensor $C$. 
Note that, since 
\be
 {\rm div}\,  {\mathcal W}= \frac {n-3}{n-2}\, C \label{eqn2018-12-13-1}
 \ee
under the identification of $\Gamma(T^*M \otimes \Lambda^2 M)$ with $\Gamma(\Lambda^2 M \otimes T^*M)$, 
 we have
\bea 
{\rm div}^4\, {\mathcal W}= \frac {n-3}{n-2}\, {\rm div}^3 \, C, \label{cotton3}
\eea
and from (\ref{bach}), 
\be (n-2) \, B= {\rm div}\, C +\mathring{\mathcal W}r_g = {\rm div}\, C +\mathring{\mathcal W}z_g . \label{bach2}
\ee
Here, $\mathring{\mathcal W}z_g$ is defined  similarly to  $\mathring{\mathcal W}r_g$. In particular, ${\rm div}^4 {\mathcal W}=0$ if and only if  ${\rm div}^3 \, C=0$.
\vskip .5pc

Hereafter, we denote curvatures $r_g, z_g, s_g$, and $D_gdf$ by
$r, z, s$, and $Ddf$, respectively, for convenience and simplicity, if there is no ambiguity.

The following shows  that the divergence of 
 the Bach tensor is related to the Cotton tensor and the traceless Ricci tensor  $z$.
  The second equation (\ref{ddivb}) follows by taking the divergence of ${\rm div}\, B$.

 \begin{prop} [cf. \cite{CC}]\label{lem2018-9-25-1} 
 The divergence of the Bach tensor is given by
$$
{\rm div}\, B (X)=\frac {n-4}{(n-2)^2} \langle i_XC, z\rangle.
$$
 Here, ${i}_{X}$ is the usual interior product to the first factor defined by
${i}_{X} {C} (Y, Z)= {C}(X,Y, Z)$ for any vector fields $X, Y$, and $Z$.
\end{prop}
By taking the divergence again, we have
\begin{corollary} \label{ddivb} 
$$
{\rm div}^2 B= \frac {n-4}{(n-2)^2} \left( \frac 12 |C|^2 +\langle {\rm div}\, C, z\rangle \right). 
$$
\end{corollary}

 For a Riemannian $n$-manifold $(M^n, g)$ and its traceless Ricci tensor $z$, $D^*Dz$ and $z\circ z$ are defined as
 $$
 D^*Dz = - \sum_{i=1}^n D^2_{E_i, E_i}z\qquad \mbox{and}\qquad 
 z\circ z(X, Y) = \sum_{i=1}^n z(X, E_i) z(Y, E_i),
 $$
respectively, for a local frame $\{E_i\}_{i=1}^n$ of $M$.
  With these notations, the divergence of the Cotton tensor, ${\rm div} \, C$, can be expressed as follows
when the scalar curvature of a Riemmanian manifold is constant.
 
\begin{lemma}\label{lem2017-12-24-2-1}
Let $(M^n, g)$ be a Riemannian manifold with a constant scalar curvature. Then,
\be
{\rm div}\, C = -D^*Dz - \frac{n}{n-2}z\circ z - \frac{s}{n-1}z 
+\mathring {\mathcal W}z +\frac{1}{n-2}|z|^2 g. \label{eqn2016-11-29-6-1} 
\ee
\end{lemma} 
\begin{proof} 
Since  $s$ is constant,  it follows from Remark 4.71 in \cite{Be}  that 
\be
{\rm div} \, C  = {\rm div} \, (d^D r )= - D^*Dr - r\circ r + {\mathring R}r.
\label{eqn2016-11-29-3-1-1} 
\ee 
Here, $r\circ r$  and  $\mathring{R}r$ are defined  similarly to  $z\circ z$ and
$\mathring{\mathcal W}r$, respectively.
From the curvature decomposition
$$ 
R = \frac{s}{2n(n-1)} g\owedge g + \frac{1}{n-2} z\owedge g +{\mathcal W},
$$
 we have
\be
{\mathring R}r = \mathring {\mathcal W}z + \frac{1}{n-2}|z|^2 g + \frac{n-2}{n(n-1)}sz 
- \frac{2}{n-2}z\circ z + \frac{s^2}{n^2}g.\label{eqn2018-9-6-1}
\ee
Therefore, we can obtain (\ref{eqn2016-11-29-6-1}) 
by substituting (\ref{eqn2018-9-6-1}) and $r = z + \frac{s}{n}g$ 
into (\ref{eqn2016-11-29-3-1-1}).
\end{proof}

 \begin{prop} \label{lems1}
Let $(M^n, g)$ be a Riemannian manifold with a constant scalar curvature. Then, we have
$${\rm div}^2 C(X) =\frac 12 \langle \tilde{i}_{X}{\mathcal W}, C\rangle 
-\frac 1{n-2}\langle i_XC, z\rangle
$$
for any vector field $X$.
\end{prop}
\begin{proof}
Note that  $C=d^Dz$ since $s$  is constant.
Thus, for a geodesic orthonormal frame $\{E_i\}_{i=1}^n$,
\bea
{\rm div}^2 C(X)&=& {\rm div}^2 (d^Dz)(X) \\
&=& D_{E_i}D_{E_k}(D_{E_k} z (E_i, X)- D_{E_i}  z(E_k, X)) \\
&=&   (D_{E_i}D_{E_k}-D_{E_k}D_{E_i})D_{E_k}z(E_i, X).
\eea
Since
$$ D^3_{X,Y,Z}\xi -D^3_{Y,X,Z}\xi= -R(X,Y)D_Z \xi +D_{R(X,Y)Z}\xi$$
for any tensor $\xi$ (cf. \cite{Be}), we have
\begin{eqnarray}
{\rm div}^2 C(X)&=&R(E_k, E_i)D_{E_k}z(E_i, X)-D_{R(E_k, E_i)E_k} z(E_i, X)\nonumber \\
&=& - D_{E_k}z(R(E_k, E_i)E_i, X) -D_{E_k}z(E_i, R(E_k, E_i) X) - z_{is}D_{E_s}z(E_i, X)\nonumber\\
&=& z_{kj}D_{E_k}z(E_j, X) +   \langle R(E_k, E_i) E_s, X)D_{E_k}z(E_i, E_s)- z_{ik}D_{E_k}
z_g(E_i, X)\nonumber\\
&=&   \frac 12 \langle R(E_k, E_i) E_s, X)C(E_k, E_i, E_s). \nonumber
\end{eqnarray}
Here, $z_{ij} = z(E_i, E_j)$. This can be written as
$$
{\rm div}^2 C(X) = \frac{1}{2} \langle {\tilde i}_X R, C\rangle.
$$
Next, multiplying  the curvature decomposition
$$
R - {\mathcal W} = \frac{s}{2n(n-1)} g\owedge g + \frac{1}{n-2} z\owedge g,
$$
by $C(E_i, E_j, E_k) = C_{ijk}$, we have
 $$
(R-{\mathcal W})(E_i, E_j, E_k, E_l) C(E_i, E_j, E_k)
  =
    \frac{1}{n-2}(z_{ik}C_{ilk} - z_{jk}C_{ljk})
    =
    - \frac{2}{n-2} z_{jk}C_{ljk}
    $$
Here, we have used the fact that  $\sum_i C(E_i, Y, E_i)=0$ for any $Y$.
\end{proof}

Before closing this section, we present examples of $n$-dimensional vacuum static spaces satisfying
${\rm div}^2 B = 0$ and ${\rm div}^4{\mathcal W}=0$, but are not Bach-flat  for $n\geq 4$. 

\begin{prop}\label{main1}
There exists an $n$-dimensional  vacuum static space satisfying ${\rm div}^2 B=0$ and ${\rm div}^4{\mathcal W}=0$ with $B\neq 0$ for $n\geq 4$.
\end{prop}
\begin{proof}
Note that, on $(M_1^n, g_1) \times (M_2^m,g_2)$ with the product metric $g=g_1+ g_2$,  a function $f\in \ker s_{g_1}'^*$ generates a function $f\in \ker s_g'^*$ if $(M_2,g_2)$ is Einstein with $\frac {s_1}{n-1} = \frac {s_2} m$ (see B.1 in \cite{laf}). 

Thus, for dimension $2n$, where $n\geq 2$, 
let  $M^{2n}=S^n(\frac a{kl})\times S^n (\frac a l)$ be a product manifold with a standard product metric $g_0$, where $l= \frac {3n(n-1)}2$, $k=\frac n{n-1}$, $a>0$.
Therefore, $(M,g_0, f)$ is a vacuum static space with $f\in \ker s_g'^*$. 
Here,  $s_g=\frac {2(2n-1)a}{3n}$, $r_{ii}= \frac {2(n-1)a}{3n^2}$ for $1\leq i\leq n$, and $r_{jj}=\frac {2a}{3n}$ for $n+1 \leq j\leq 2n$. It is also easy to see that $C=0$ since $g=g_1+g_2$ is a standard product metric with Einstein metrics $g_i$, $i=1,2$.
Thus, choosing a local frame $\{E_1, \cdots, E_n, E_{n+1}, \cdots, E_{2n}\}$ on $M^{2n}$ such that
$\{E_1, \cdots, E_n\}$ and $\{E_{n+1}, \cdots, E_{2n}\}$
are local frames on $S^n(\frac a{kl})$ and $S^n (\frac a l)$, respectively, 
we have
$$ {\mathcal W}(E_1, E_i, E_1, E_i)= \frac {a}{3n(n-1)}$$
for $i=2,..., n$, and
$$ {\mathcal W}(E_1, E_{j}, E_1, E_{j})=- \frac {a}{3n^2}$$
for $j=n+1,..., 2n$.
Thus, since $C=0$, by (\ref{bach2}) we have
$$ 2(n-1)B(E_1, E_1)= \mathring{\mathcal W}r(E_1, E_1)= -\frac {2a^2}{9n^3}\neq 0.$$

Now, for dimension $2n+1$, $n\geq 2$, it is easy to see that  $(S^n\times S^{n+1}, g, f)$ with a standard product metric $g$ satisfies $f\in \ker s_g'^*$. Similarly, $C=0$. Thus,  we have $s_g=2n^2$. Therefore,
$${\mathcal W}(E_1, E_i, E_1, E_i)= \frac {n+1}{2n-1}$$
for $2\leq i\leq n$,
and
$$ {\mathcal W}(E_1, E_{j},E_1, E_{j})= - \frac {n-1}{2n-1}$$
for $n+1 \leq j \leq 2n$.
Since $C=0$, we have
$$(2n-1) \, B(E_1, E_1)=\mathring{\mathcal W}r(E_1, E_1)= -\frac {n^2-1}{2n-1}\neq 0.
$$
\end{proof}

\section{Harmonicity}

In this section, we will give a proof of Theorem~\ref{main3}.
It suffices to prove that the Cotton tensor  $C$ vanishes in view of  (\ref{eqn03-1}).
First, by applying $d^D$ to both sides of (\ref{static1}), or
\be
 fz_g =Ddf +\frac {sf}{n(n-1)}g,\label{eqn2018-12-13-2}
 \ee
we obtain the following result.

\begin{lemma}\label{lem2018-9-5-3} 
Let $(M, g, f)$ be a non-trivial vacuum static space.  Then, 
\be
f \, C  =  \tilde{i}_{\nabla f}{\mathcal W}-(n-1)\, T. \label{eqn09}
\ee
\end{lemma}
\begin{proof} 
From  (\ref{static1}),
\be
f d^Dz = d^DDdf + \frac{s}{n(n-1)}df \wedge g - df \wedge z. \label{eqn2018-9-5-6}
\ee
Note that 
$$
d^DDdf = {\tilde i}_{\n f}R,
$$
where ${\tilde i}_{\n f}R$ is defined similarly to ${\tilde i}_{\n f}{\mathcal W}$.
Next, from the curvature decomposition
$$ 
R = \frac{s}{2n(n-1)} g\owedge g + \frac{1}{n-2} z\owedge g +W,
$$
we have
$$
{\tilde i}_{\n f}R = {\tilde i}_{\n f}{\mathcal W} - \frac{1}{n-2}i_{\n f}z \wedge g 
- \frac{s}{n(n-1)}df \wedge g - \frac{1}{n-2}df\wedge z.
$$
Substituting these into (\ref{eqn2018-9-5-6}) and using the definition of $T$, our Lemma follows.
\end{proof}

\begin{corollary}\label{cor2018-3-26-1}
On a vacuum static space, we have
\bea
{\rm div}^2 C(\nabla f)= \frac 12 f |C|^2 +\langle i_{\nabla f}C, z\rangle.\label{eqs3}\eea
\end{corollary}
\begin{proof}
Note that
\be
\langle T, C\rangle = \frac 1{n-2}\langle df \wedge z , C\rangle =\frac 2{n-2}\langle i_{\nabla f}C, z\rangle.\label{eqn2018-12-13-6}
\ee
Thus, from Proposition~\ref{lems1} and Lemma~\ref{lem2018-9-5-3}, we have
$$
{\rm div}^2 C(\nabla f)= \frac 12 f|C|^2 +\frac {n-1}2 \langle T, C\rangle -\frac 1{n-2} \langle i_{\nabla f}C, z\rangle =\frac 12 f |C|^2 +\langle i_{\nabla f}C, z\rangle.
$$

\end{proof}

For any real numbers $t$ and $t'$, let $M_{t,t^{'}}=\{x\in M\,\vert\,  t \le f(x) \le t^{'} \}$ 
and $\Gamma_t=\{ x\in M\, \vert\ f(x)= t\}$. 
\begin{lemma}\label{lem2} 
Assume that ${\rm div}^2 B=0$ for $n\geq 5$ or $\frac 12 |C|^2=-\langle {\rm div}\, C, z\rangle$ for $n=4$ on a vacuum static space $M$ with compact level sets of $f$. Then, we have
$$\int_{\Gamma_t}\frac 1{|\nabla f|} \langle i_{\nabla f}C, z\rangle =\int_{\Gamma_{t^{'}}} \frac 1{|\nabla f|} \langle i_{\nabla f}C, z\rangle $$
for regular values $t$ and $t^{'}$ of $f$ with $t<t'$.
\end{lemma}
\begin{proof}
By Corollary~\ref{ddivb} and the assumption that ${\rm div}^2 B=0$ for $n\geq 5$, or $\frac 12 |C|^2=-\langle {\rm div}\, C, z\rangle$ for $n=4$, we have
\be \frac 12 |C|^2=-\langle {\rm div}\, C, z\rangle. \label{eqs2}\ee
For an orthonormal frame $\{E_i\}_{i=1}^n$, we compute
$$ \langle {\rm div}\,  C, z\rangle = {\rm div} ( C(\cdot, E_i, E_j)z_{ij})- C_{ijk}D_{E_i}z_{ij}= {\rm div} ( C(\cdot, E_i, E_j)z_{ij})- \frac 12 |C|^2.$$
Thus, from (\ref{eqs2}), we have
$$  {\rm div} ( C(\cdot, E_i, E_j)z_{ij}) =0$$
on $M$.
This implies that 
$$0=\int_{M_{t,t^{'}}}{\rm div} ( C(\cdot, E_i, E_j)z_{ij})   =\int_{\partial M_{t, t^{'}} }\frac 1{|\nabla f|}\langle i_{\nabla f} C, z\rangle.$$
\end{proof}

 Now, we are ready to prove Theorem~\ref{main3} and \ref{4dim1}, which state that the vanishing of complete divergences of the Bach tensor and Cotton tensor on a vacuum static space implies that $(M, g)$ has
 harmonic curvature for $n\geq 5$, if the level sets of $f$ are compact. \vskip .5pc

For regular values $t$ and $t'$ of $f$ with $t<t'$, from Corollary~\ref{cor2018-3-26-1}, we have
\bea
\int_{M_{t,t'}} {\rm div}^3 C &=&\int_{\Gamma_{t'}} {\rm div}^2 C(N) -\int_{\Gamma_{t}} {\rm div}^2 C(N) \\
&=& \frac {t'}2 \int_{\Gamma_{t'}} \frac {|C|^2}{|\nabla f|}  -\frac t2 \int_{\Gamma_t} \frac {|C|^2}{|\nabla f|}.\eea
Here, $N = \frac{\n f}{|\n f|}$, and we used the result of Lemma~\ref{lem2} in the last equality. 
 Therefore,  we have
$$ t \int_{\Gamma_t} \frac {|C|^2}{|\nabla f|} ={t'} \int_{\Gamma_{t'}} \frac {|C|^2}{|\nabla f|}.$$
By taking $t'=0$, we may conclude that $C=0$ on $\Gamma_t$ for all regular values $t$ of $f$ with $t<0 $. Similarly, $C=0$ on $\Gamma_t$ for regular values of $f$ with $t>0$. Hence, we may conclude that $C=0$ on all of $M$ by continuity. In other words, $M$ has harmonic curvature. This proves our theorems.
\vskip .5pc

\section{Uniqueness of black holes and Besse conjecture}
In this section, as applications of Theorem~\ref{main3} and \ref{4dim1}, we will prove Theorem~\ref{main2} and \ref{main2-100}. First, we prove that multiple black holes do not exist in an $n$-dimensional
 static vacuum spacetime under the vanishing of complete divergence of the Bach tensor and Weyl tensor (Theorem~\ref{main2}).
The $n$-dimensional static vacuum Einstein equation is given by (\ref{static:eq1}), or
\bea
\begin{array}{ll}
h\, r_g = D_gdh, \\
\Delta_g h = 0.
\end{array}
\eea
In particular, the scalar curvature ${s}_g$ of ${g}$ vanishes. In fact, a non-trivial solution $(g,h)$ of (\ref{static:eq1}) is a vacuum static space with zero scalar curvature.
Note that solutions to these equations constitute
a Ricci-flat $(n+1)$-dimensional manifold $\overline{M}$ of the form $\overline{M}=M\times_hS^1$
or $\overline{M}=M\times_h {\mathbb R}$, with a Riemannian or Lorentzian metric of the form
$$ \overline{g}=g\pm h^2 dt^2. $$
If the manifold $(M,g)$ satisfying (\ref{static:eq1}) is  geodesically complete, then $h$ is known to be a constant function (\cite{An1}). A vacuum static space $(M, g, h)$  is said to be {\it $h$-weakly harmonic} with a function $h$ if the Ricci
 curvature $r$ satisfies $d^D r (\n h, \cdot, \n h) = 0$.

\begin{theorem}[\cite{grg}]\label{thm2012-12-29-1}
Assume that $(M, g, h)$ is a vacuum static space with vanishing scalar curvature and $h^{-1}(0)=\partial M$ is compact.
If $(M,g,h)$ has $h$-weakly harmonic curvature, then multiple black holes do not exist in $M$.  \label{thm:main}
\end{theorem}
In particular, if $M$ has harmonic curvature, then the same result as Theorem~\ref{thm2012-12-29-1} holds. 
\vskip .5pc
For the proof of Theorem~\ref{main2}, it suffices to prove that $C=0$ on $M$. We simply follow the proof of Theorem~\ref{main3} and \ref{4dim1} to show that $(M,g)$ has harmonic curvature. Then, from Theorem~\ref{thm2012-12-29-1}, we may conclude that $M$ has no multiple black holes, given that $|\nabla h|^2$ is constant on each level set of $h$. 
\vskip .5pc

Secondly, we prove the Besse conjecture under the vanishing of complete divergence of the Bach tensor and Weyl tensor (Theorem~\ref{main2-100}). The proof is similar to the case of vacuum static spaces. Let $(g,f)$ be a non-trivial solution of (\ref{eqn22}), or
$$ s_g'^*(f)=z.$$
Note that, if $C=0$, we have the following result.
\begin{theorem} [\cite{ych2}] Let $(g,f)$ be a non-trivial solution of (\ref{eqn22}) on an $n$-dimensional compact manifold $M$, $n\geq 4$. If $M$ has harmonic curvature, then the Besse conjecture holds.
\end{theorem}
Therefore, it suffices to prove that $C=0$ on all of $M$. 
For the proof, we apply $d^D$ to (\ref{eqn22}) to obtain
$$ (1+f)\, C =\tilde{i}_{\nabla f}{\mathcal W}-(n-1)T,$$
where $T$ is defined as (\ref{defnt}) (compare with (\ref{eqn2018-9-5-6})). Moreover, we have
$$ {\rm div}^2 C(\nabla f)= \frac 12 (1+f)|C|^2+\langle i_{\nabla f}C, z\rangle$$
(compare with Corollary~\ref{cor2018-3-26-1}). If ${\rm div}^2 B=0$, then for regular values $t$ and $t'$ of $f$ with $t<t'$, 
$$ \int_{\Gamma_{t}} \frac 1{|\nabla f|}\langle i_{\nabla f}C, z\rangle = \int_{\Gamma_{t'}}\frac 1{|\nabla f|}\langle i_{\nabla f}C, z\rangle$$
(compare with Lemma~\ref{lem2}). Therefore, for regular values $t$ and $t'$ of $f$ with $t<t'$, we have
\bea
\int_{M_{t,t'}} {\rm div}^3 C &=&\int_{\Gamma_{1+t'}} {\rm div}^2 C(N) -\int_{\Gamma_{1+t}} {\rm div}^2 C(N) \\
&=& \frac {1+t'}2 \int_{\Gamma_{1+t'}} \frac {|C|^2}{|\nabla f|}  -\frac {1+t}2 \int_{\Gamma_{1+t}} \frac {|C|^2}{|\nabla f|}.\eea
By taking $t'=-1$ with the assumption that ${\rm div}^4{\mathcal w}=0$ and ${\rm div}^2 B=0$ for $n\geq 5$, or $\frac 12 |C|^2=-\langle {\rm div}\, C, z\rangle$ for $n=4$, we may conclude that, for $t\neq -1$, 
$$ \frac {1+t}2 \int_{\Gamma_{1+t}} \frac {|C|^2}{|\nabla f|} =0.$$
This implies that $C=0$ on each regular level set of $f$, or on all of $M$ by continuity, since critical points of $f$ on $M$ have measure zero (see Proposition 2.2 of \cite{ych2}).
Consequently, we may conclude that $(M,g,f)$ is Einstein or isometric to a standard sphere.

\section{Radially Bach-flat vacuum static spaces}

In this section, we consider vacuum static spaces $(M,g, f)$ satisfying  (\ref{static1})
with radial Bach flatness. 
The notion of radial Bach flatness originated
from \cite{p-w}, in which Petersen and Wiley defined the notion of a radially flat curvature, a radially
 and conformally flat curvature, or radial Ricci flatness for gradient Ricci solitons.
As mentioned in the introduction, Qing and Yuan (\cite{QY}) classified Bach-flat vacuum static spaces $(M, g, f)$ with compact level sets of $f$. For the proof, they introduced a
$3$-tensor that is identical to $fT$ in our case, which is defined in Lemma 3.1. By showing an integral identity (Proposition 2.3), they proved that the tensor $T$ must be vanishing
for Bach-flat vacuum static spaces. We can also show the same identity (Lemma~\ref{lem2018-2-10-2}), and
we will give a proof for self-containedness.
In view of those identities, we can see that, for a vacuum static space, 
the vanishing of $T$ is implied only by the vanishing of $i_{\n f}B$.
We also prove the converse under slightly weaker conditions. That is, for a vacuum static space
 satisfying ${\rm div}^2 B = 0$,
the condition ${\rm div}^3 T=0$ implies $(M, g)$ is Bach-flat.
 In view of Proposition 2.4, we can see that the vanishing of the Cotton tensor does not imply Bach flatness nor
 the vanishing of the tensor $T$ for a vacuum static space, even though it satisfies ${\rm div}^2 B = 0$.

First, we present  various properties of the Bach tensor $B$,  tensor $T$, and their divergences for vacuum static spaces $(M, g, f)$. As an application, we will prove
a rigidity for compact vacuum static spaces with radial Bach flatness by using the maximum principle.

\begin{lemma}\label{lem2018-2-10-1}
Let $(g, f)$ be a solution of (\ref{static1}). Then,
$$
{\rm div} ({\tilde i}_{\n f} {\mathcal W})  =  \frac{n-3}{n-2} \widehat C 
- f {\mathring {\mathcal W}}z,
$$
where ${\widehat C}$ is a $2$-tensor defined as
$$
{\widehat C}(X, Y) = C(Y, \n f, X)
$$
for any vectors $X, Y$. 
\end{lemma}
\begin{proof}
Let $\{E_i\}$ be a local geodesic  frame on $M$. From (\ref{eqn2018-12-13-1})
 and (\ref{eqn2018-12-13-2}), together with the fact that  ${\rm tr}_{14} {\mathcal W} =0$,
we have
\bea
 {\rm div} ({\tilde i}_{\n f} {\mathcal W}) (X, Y) 
 &=&   E_i({\mathcal W}(E_i, X, Y, \n f))\\
 &=&
 D_{E_i}{\mathcal W}(E_i, X, Y, \n f) + {\mathcal W}(E_i, X, Y, D_{E_i}d f)\\
 &=&
 {\rm div} {\mathcal W}(X, Y, \n f) + Ddf(E_i, E_j)  {\mathcal W}(E_i, X, Y, E_j)\\
    &=&
       \frac{n-3}{n-2}C(Y, \n f, X) - f {\mathring {\mathcal W}}z (X, Y).
 \eea 

\end{proof}

By a straightforward computation, we obtain the following.
\begin{lemma}\label{lem2018-9-25-3}
Let $(M, g, f)$ be a non-trivial vacuum static space. Then,
\be 
|T|^2 = \frac{2}{n-2}\langle i_{\n f} T, z\rangle\label{eqn1-1}
\ee
and
\be
|T|^2 = \frac{2}{(n-2)^2}|\n f|^2 \left(|z|^2 - \frac{n}{n-1}|i_{N} z|^2\right).\label{eqn2} 
\ee 
\end{lemma}


\begin{lemma}\label{lem2018-2-10-2} 
Let $(M, g, f)$ be a complete vacuum static space with compact level sets of $f$.  Then, 
$$ 
\int_{M_{t,t'}} f^2 B(\nabla f, \nabla f) =-\frac {n-1}2 \int_{M_{t,t'}} f^2 |T|^2
$$ 
for any regular values $t$ and $t'$ with $t<t'$. 
\end{lemma}
\begin{proof} 
From Lemma~\ref{lem2018-9-5-3}  and Lemma~\ref{lem2018-2-10-1},
\bea 
f^2 {\rm div} C &=& {\rm div} (f^2 C) - 2f i_{\n f}C\\
&=&
{\rm div} ( f\tilde{i}_{\nabla f}{\mathcal W}) - (n-1) {\rm div} (fT) - 2f i_{\n f} C\\ 
&=& 
i_{\n f}{\tilde i}_{\n f}{\mathcal W} + f {\rm div} \left({\tilde i}_{\n f}{\mathcal W}\right)
- (n-1)\,  {\rm div} (fT) - 2f\,  i_{\n f} C\\ 
&=& 
{\mathcal W}(\nabla f , \cdot , \cdot, \nabla f) + \frac {n-3}{n-2}\, f {\widehat C}
-  f^2\mathring{\mathcal W}z
  - (n-1)\,  {\rm div} (fT) - 2 f i_{\n f}C. 
\eea 
Therefore,  from  (\ref{bach2}), 
\bea 
(n-2)f^2B &=& f^2\,  {\rm div} C  + f^2\, \mathring{\mathcal W} z \\ 
&=& 
{\mathcal W}(\nabla f , \cdot , \cdot, \nabla f) +\frac {n-3}{n-2}f {\widehat C}
  - (n-1) {\rm div} (fT) - 2 f i_{\n f}C.
\eea 
In particular, 
$$ 
f^2 B(\nabla f, \nabla f)= - \frac {n-1}{n-2}\, {\rm div} (fT)(\nabla f, \nabla f). 
$$ 
Let  $\{E_1, \cdots, E_{n-1},  N= \frac{\n f}{|\n f|}\}$ be a local frame around a regular point of $f$.
Then,
\bea 
{\rm div} (fT)(\nabla f, \nabla f)
&=&
{E_i}(fT(E_i, \nabla f, \nabla f)) - fT(E_i, D_{E_i}df, \nabla f) 
 - fT(E_i, \nabla f, D_{E_i}df)\\ 
&=&
{E_i}(fT(E_i, \nabla f, \nabla f)) - f^2 T(E_i, \nabla f, E_k)z(E_i, E_k)\\ 
&=&
\mbox{div}(fT(\cdot, \nabla f, \nabla f)) +\frac {n-2}2 f^2 |T|^2. 
\eea 
In the second equality, we used the fact that $T(E_i, E_j, \n f) =0$, together with 
${\rm tr}_{12}T = {\rm tr}_{13}T = 0$ and (\ref{eqn2018-12-13-2}).
In the last equality, we used (\ref{eqn1-1}).
Therefore, by the divergence theorem we have
$$
-\frac {n-2}{n-1}\int_{M_{t,t'}} f^2 B(\nabla f, \nabla f) = \frac {n-2}2\int_{M_{t,t'}} f^2|T|^2=0, $$
implying our conclusion.
 \end{proof}

Next, we will show that, for a non-trivial vacuum static space $(M, g, f)$ satisfying ${\rm div}^2 B=0$,
the vanishing of $B$ is equivalent to $B(\n f, \n f) = 0$.
First, from the definition of $T$, we have, on each level hypersurface
$f^{-1}(t)$ for a regular value $t$ of $f$, 
\be 
i_{\n f} T 
=
\frac{1}{n-2}|\nabla f|^2 z + \frac{1}{(n-1)(n-2)} z(\n f, \n f) g  
= 
\frac{|\nabla f|^2}{n-2}\left( z + \frac{\a }{n-1} g\right),\label{eqn2017-5-26-3-1} 
\ee 
where $\a:= z(N, N) = z \left(\frac{\n f}{|\n f|}, \frac{\n f}{|\n f|}\right)$. Note that the function $\a$ is defined only on the set $M \setminus {\rm Crit}(f)$, where
${\rm Crit}(f)$ is the set of all critical points of $f$. However, since $|\a|\le |z|$, $\a$
can be extended to a $C^0$ function on the whole manifold $M$
 (for more details, see \cite{ych2}). Therefore, when $T=0$,  it follows from (\ref{eqn2017-5-26-3-1}) that
\be
z (E_i, E_j) = - \frac{\a}{n-1}\delta_{ij} \label{eqn2018-3-22-1}
\ee
for a local frame $\{E_1, \cdots, E_{n-1},  N= \frac{\n f}{|\n f|}\}$.
Moreover, by substituting the triple $(\n f, X, \n f)$ into $T$  for a vector field $X$ with 
$\langle X, \n f\rangle =0$, we obtain
$$ 
|\nabla f|^2 z(X, \n f) - \frac{1}{n-1}z(\n f, X)|\nabla f|^2 = 0, 
$$ 
which implies 
\bea
z(\n f, X) = 0 
\eea
 for any vector $X$ with $X \perp \n f$. Thus,
the traceless Ricci tensor $z$ can be expressed as a diagonal matrix and
\bea 
|z|^2 = \frac{n}{n-1}\a^2. 
\eea
We can also derive this identity from (\ref{eqn2}) in Lemma~\ref{lem2018-9-25-3}.

\begin{lemma}\label{lem2018-2-14-11}
Let   $(M, g, f)$ be an $n$-dimensional vacuum static space. Then,
\be
(n-2) f B = -i_{\n f}C + \frac{n-3}{n-2}{\widehat C} - (n-1)\,  {\rm div}\,  T.\label{eqn2017-6-12-11}
\ee
\end{lemma}
\begin{proof}
By taking the divergence of both sides in (\ref{eqn09}) and using Lemma~\ref{lem2018-2-10-1}, we obtain
$$
f\left({\mathring {\mathcal W}}z + {\rm div}\,  C\right) = -i_{\n f}C +   
\frac{n-3}{n-2}{\widehat C} - (n-1)\, {\rm div} \, T.
$$
Therefore, (\ref{eqn2017-6-12-11}) follows from (\ref{bach2}).
\end{proof}

\begin{lemma}\label{lem2018-2-10-3}
We have
\be
{\rm div}^2 \, T (X) = \frac{1}{n-2}\,  f\langle i_X C, z\rangle + \langle i_X T, z\rangle\label{eqn2018-2-9-1}
\ee 
for any vector $X$.
\end{lemma}
\begin{proof}
First, note that
\be
\langle i_X {\tilde i}_{\n f}{\mathcal W}, z\rangle = - {\mathring {\mathcal W}}z(\n f, X).\label{eqn2018-2-14-13}
\ee
Thus, by taking the divergence of both sides of (\ref{eqn2017-6-12-11}) and using the fact that
$\langle z, {\tilde i}_X C\rangle =0$ and ${\rm div}\, i_{\n f}C = - i_{\n f}{\rm div}C$, 
together with (\ref{eqn2018-12-13-2}), we obtain
\bea
(n-2)i_{\n f}B + (n-2) f {\rm div} B =
i_{\n f} {\rm div} C + \frac{n-3}{n-2} {\rm div} {\widehat C}  - (n-1)\,  {\rm div}^2\,  T.
\eea
Thus,
\bea
 (n-1){\rm div}^2 T(X) = -(n-2) B(\n f, X) - (n-2) f {\rm div}\,  B(X) + {\rm div}\,  C(\n f, X) 
 +  \frac{n-3}{n-2} {\rm div}\,  {\widehat C}(X).
\eea
Since from (\ref{bach2}), ${\rm div}\, C$ is a symmetric $2$-tensor, by using (\ref{eqn2018-12-13-2}), 
one can also compute
\bea
{\rm div}\,  {\widehat C} (X) = f \langle i_X C, z\rangle.\label{eqn2018-2-9-2-10}
\eea
By substituting these in Proposition~\ref{em2018-9-25-1} and replacing $B$ with (\ref{bach2}), we obtain
$$
 (n-1){\rm div}^2 T(X) = \frac{1}{n-2}f \langle i_X C, z\rangle 
 - {\mathring {\mathcal W}}z(\n f, X).
 $$
 Finally, from Lemma~\ref{lem2018-9-5-3}  together with (\ref{eqn2018-2-14-13}), we have
 $$
 (n-1) \langle i_XT, z\rangle + f \langle i_X C, z\rangle 
 = \langle i_X {\tilde i}_{\n f}{\mathcal W}, z\rangle 
 = - {\mathring {\mathcal W}}z(\n f, X).
$$
Hence,
 $$
 (n-1){\rm div}^2 T\,  (X) =
\frac{n-1}{n-2} f \langle i_X C, z\rangle + (n-1)\langle i_X T, z\rangle.
$$
\end{proof}

\begin{lemma}\label{lem2018-1-20-11-1} 
 Let $(M, g, f)$ be an $n$-dimensional vacuum static space. If $T = 0$, then
\begin{itemize} 
\item[(1)] $\langle i_X C, z \rangle = 0$ for any vector $X$; therefore, $ {\rm div} \, B = 0$.
\item[(2)] $\langle {\mathcal W}_N, z\rangle = 0$, where ${\mathcal W}_N$ is defined by
${\mathcal W}_N(X, Y) = {\mathcal W}(N, X, N, Y)$ with $N = \frac{\n f}{|\n f|}$.
\end{itemize} 
\end{lemma}
\begin{proof} 
From Lemma~\ref{lem2018-2-10-3},
we have $\langle i_X C, z \rangle = 0$ for any vector $X$; therefore, from Propostion~\ref{lem2018-9-25-1}  and
(\ref{eqn2018-12-13-6}),
$$ 
{\rm div} B = 0\quad \mbox{and}\quad \langle T, C\rangle = 0. 
$$ 
In particular, since $\langle i_{\n f}C , z\rangle =0$, from (\ref{eqn09}), 
\bea 
0 = (n-1)\langle i_{\n f}T, z\rangle = \langle i_{\n f}{\tilde i}_{\n f}{\mathcal W}, z\rangle 
= - |\n f|^2 \langle {\mathcal W}_N, z\rangle.\label{eqn2018-2-10-9} 
\eea
\end{proof}

\begin{prop}\label{prop2018-1-20-11}
 Let $(M, g, f)$ be an $n$-dimensional vacuum static space satisfying ${\rm div}^2 B=0$.
 If $T = 0$, then $B=0$ and $C=0$.
\end{prop}
\begin{proof} 
From Lemma~\ref{lem2018-9-5-3}, $T=0$ implies that
\be 
f C = {\tilde i}_{\n f} {\mathcal W}. \label{eqn2018-2-10-7-1} 
\ee 
Therefore,
\bea 
C(X, Y, \n f) = 0 \label{eqn2018-2-10-5-1} 
\eea 
for any vectors $X$ and $Y$. Since the cyclic summation of $C$ vanishes,
 we have 
$$ 
C(Y, \n f, X) + C(\n f, X, Y) = 0. 
$$ 
In other words,
\bea 
i_{\n f}C + \widehat C = 0.\label{eqn2018-2-10-6-1} 
\eea 
By taking the divergence of (\ref{eqn2018-2-10-7-1}) and 
using Lemma~\ref{lem201 -2-10-1}, we have 
\bea 
f\,  {\rm div}\, C + f \, {\mathring {\mathcal W}}z = - \frac{2n-5}{n-2}\, i_{\n f}C.\label{eqn2018-2-10-8-1} 
\eea
Since ${\mathring {\mathcal W}}z(N, N) = \langle {\mathcal W}_N, z\rangle=0$,
from Lemma~\ref{lem2018-1-20-11-1}, we have 
\bea
f \,{\rm div} \, C(N, N) = 0.\label{eqn2018-2-10-10-1} 
\eea
Finally,  from Corollary~\ref{ddivb} and (\ref{eqn2018-3-22-1}), together with  
$\displaystyle{\sum_{j=1}^{n-1}C_{ijj;i} = - C_{inn;i}}$, we have
\bea 
\frac{1}{2}|C|^2 &=& - \langle {\rm div} C, z\rangle = - C_{ijk;i}z_{jk}=
 \frac{\a}{n-1} \sum_{i=1}^n \sum_{j=1}^{n-1}C_{ijj;i} - \a \sum_{i=1}^n  C_{inn;i}\\
&=&
- \frac{n\a}{n-1} \sum_{i=1}^n  C_{inn:i} = - \frac{n\a}{n-1} {\rm div} \, C(N, N) =0. 
\eea 
Thus, $C = d^D z = 0$; therefore, $B = 0$ from Lemma~\ref{lem2018-2-14-11}.
\end{proof}

\begin{corollary} \label{cor122}
Let $(M, g, f)$ be an $n$-dimensional vacuum static space satisfying ${\rm div}^2 B=0$. Then,
$B(\n f, \n f) = 0$ if and only if $B=0$.
\end{corollary}
\begin{proof}
The proof follows from Lemma~\ref{lem2018-2-10-2}  and
Proposition~\ref{prop2018-1-20-11}.
\end{proof}

{Next, we will show that, for a vacuum static space $(M, g, f)$ with
 $i_{\n f}B = 0$, the square norm of the traceless
 Ricci tensor, $|z|^2$, attains its maximum on the set $f^{-1}(0)$.
 As a corollary, if $|z|^2$ attains its local maximum on the set $M\setminus f^{-1}(0)$,
$(M, g)$ is Einstein and isometric to a standard sphere when $M$ is compact.}

\vskip .5pc

To show this property, we first compute  the divergence of the tensor $T$.

\begin{lemma}\label{lem2017-12-25-10-100-1} 
Let $(M, g, f)$ be an $n$-dimensional  vacuum static space. 
Then,
\bea 
(n-1)(n-2)\, {\rm div} \, T = - \frac{n-2}{n-1} sfz + (n-2)D_{\n f} z - \widehat C 
- n f z\circ z + f |z|^2 g.\label{eqn2017-12-7-2-1-100}
\eea 
\end{lemma} 
\begin{proof} 
By using (\ref{eqn2018-9-5-1}) and (\ref{eqn2018-12-13-2}), we can compute
\bea
{\rm div} (df\wedge z) = - \frac{sf}{n}z + D_{\n f}z - f z\circ z.\label{eqn2018-9-10-1}
\eea
From the fact that $\delta z = 0$ and from (\ref{eqn2018-12-13-2}), we can obtain
\bea
{\rm div}  (i_{\n f}z \wedge g) = - D_{\n f}z - {\widehat C} + f|z|^2 g - fz\circ z 
+ \frac{sf}{n(n-1)} z. \label{eqn2018-9-10-2}
\eea
 From the definition of the tensor $T$, we have
 \bea
 (n-2)\, {\rm div}\,  T &=& {\rm div} (df\wedge z)  + \frac{1}{n-1} {\rm div} (i_{\n f}z \wedge g)\\
 &=&
- \frac{n-2}{(n-1)^2} sfz + \frac{n-2}{n-1}\, D_{\n f}z - \frac{n}{n-1}f \, z\circ z
 - \frac{1}{n-1}\, {\widehat C} + \frac{1}{n-1}f|z|^2 g.
 \eea
\end{proof}

We can also show the following by using (\ref{eqn09}).

\begin{lemma}\label{lem2017-12-25-10-100} 
Let $(M^n, g, f)$ be a non-trivial vacuum static space. Then,
\bea 
(n-1)(n-2)\, {\rm div}\,  T &=& - (n-2) f D^*Dz - n f z\circ z - \frac{n-2}{n-1} s f z 
+ f |z|^2 g \label{eqn2017-12-7-3-1}\\ 
&&
\quad + (n-3){\widehat C} - (n-2)i_{\n f}C - 2(n-2) f {\rm div}\,  C. \nonumber 
\eea 
\end{lemma}
\begin{proof} 
From (\ref{eqn09}) together with Lemma~\ref{lem2017-12-24-2-1}
 and Lemma~\ref{lem2018-2-10-1}, we have
\bea
(n-1){\rm div} \, T &=& {\rm div} ({\tilde i}_{\n f}{\mathcal W}) - i_{\n f}C - f{\rm div}\,  C\\
&=&
- f {\mathring W}z + \frac{n-3}{n-2} {\widehat C} - i_{\n f}C - f {\rm div} \, C\\
 &=& 
- f D^*Dz - \frac{n}{n-2} f z\circ z - \frac{s}{n-1}f z + \frac{1}{n-2}f|z|^2 g \\
 &&
 \quad  + \frac{n-3}{n-2}\,  {\widehat C} - i_{\n f}C - 2 f\,  {\rm div}\,  C. 
 \eea 
 
\end{proof}

Comparing Lemma~\ref{lem2017-12-25-10-100-1} with
Lemma~\ref{lem2017-12-25-10-100}, we obtain the following.

\begin{corollary} \label{cor2018-9-10-3}
Let $(M, g, f)$ be an $n$-dimensional  vacuum static space. Then, 
\be 
f D^*Dz + D_{\n f} z 
= {\widehat C} - i_{\n f}C - 2 f \, {\rm div} \, C.\label{eqn2018-2-21-10} 
 \ee
\end{corollary}

\begin{lemma}\label{prop2018-9-10-6}
Let $(M, g, f)$ be an $n$-dimensional  vacuum static space satisfying ${\rm div}^2 B=0$. If $i_{\n f}B = 0$, 
then the function $|z|^2$ attains its maximum on the set $f^{-1}(0)$.
\end{lemma}
\begin{proof}
From Lemma~\ref{lem2018-2-10-2} and Proposition~\ref{prop2018-1-20-11}, $T=0=C$. Therefore,
$$
f D^*Dz + D_{\n f}z =0
$$
from Corollary~\ref{cor2018-9-10-3}.
Since $\displaystyle{\frac{1}{2}\Delta |z|^2 = -\langle D^*Dz, z\rangle + |Dz|^2}$,
we have
\be
\frac{1}{2}\, f\Delta |z|^2 - \frac{1}{2}\n f(|z|^2) = f |Dz|^2.\label{eqn2018-9-10-4}
\ee
This shows that, on the set $f^{-1}(0)$,
$$
\n f(|z|^2 ) = 0.
$$
Now, on the set $\{f \ge \epsilon\}$, for a sufficiently small positive real number $\epsilon >0$,
 (\ref{eqn2018-9-10-4}) can be expressed as
$$
\frac{1}{2}\Delta |z|^2 - \frac{1}{2f}\n f(|z|^2) = |Dz|^2.
$$
By applying the maximum principle, we obtain
$$
\max_{f\ge \epsilon} |z|^2 = \max_{f = \epsilon}|z|^2.
$$
By letting $\epsilon \to 0$, we finally obtain
$$
\max_{f\ge 0} |z|^2 = \max_{f = 0}|z|^2.
$$
Similarly, by applying the maximum principle on the set $\{f \le -\epsilon\}$, we have
$$
\max_{f\le 0} |z|^2 = \max_{f = 0}|z|^2.
$$
Hence we prove our Lemma. 
\end{proof}

Now, we are ready to prove Theorem~\ref{main12}.

From Lemma~\ref{prop2018-9-10-6}, $|z|^2$ should be constant on $M$; therefore, $\langle D^*Dz, z\rangle = |Dz|^2$.
From Lemma~\ref{lem2017-12-25-10-100}, we have
\be
 (n-2)|Dz|^2 + n \langle z\circ z, z\rangle +\frac{n-2}{n-1}s |z|^2 =0.\label{eqn2018-9-10-9}
 \ee
  Fix a point $p \in M$ and choose a local frame that diagonalizes $z$ at $p$. Then,
 (\ref{eqn2018-9-10-9}) becomes
 $$
  (n-2)|Dz|^2 + n|z|^3 + \frac{n-2}{n-1}s|z|^2 = 0
 $$
 at the point $p$, which implies that $|z|(p) =0$. 
 Since   the point $p$ is arbitrary, we can conclude that $z = 0$ on $M$.
 Finally, from (\ref{eqn2018-12-13-2}),
 $$
 Ddf + \frac{sf}{n(n-1)}g = 0,
 $$
 which implies that $(M, g)$ is isometric to a standard sphere, as shown by Obata (\cite{Ob}).

\vskip .5pc

Finally, we will show some rigidity results for vacuum static spaces with complete divergence 
of the tensor $T$. That is, we prove that, if  $(M^n, g, f)$ is a non-trivial vacuum static space 
with  compact level sets of $f$ and if  ${\rm div}^2 B=0$ and  ${\rm div}^3 T=0$, 
then $T=0$. Therefore, $(M, g)$ must be Bach-flat.

To show this, we first need the following property on the complete divergence of $T$.

\begin{lemma} \label{lem3}
$$ 
{\rm div}^3 T = \frac {2(n-2)}{n-4} {\rm div} \, B(\nabla f) 
+ \frac {n-2}{n-4} f\,  {\rm div}^2 B +\langle {\rm div} \, T, z\rangle.
$$
\end{lemma}
\begin{proof}
By substituting the equation in Proposition~\ref{lem2018-9-25-1} into  (\ref{eqn2018-2-9-1}), we obtain
\bea
 {\rm div}^2 T(X)  
&=&  \frac {n-2}{n-4} f \, {\rm div}\,  B(X) + \langle i_XT, z\rangle.
\eea
By taking the divergence of this equation again,  we have
$$
{\rm div}^3 T=\frac {n-2}{n-4}\,  {\rm div}\,  B(\nabla f)  +\frac {n-2}{n-4} f\,  {\rm div}^2 B 
+ \langle {\rm div} \, T, z\rangle +\frac 12 \langle T, C\rangle.$$
Our Lemma follows from Proposition~\ref{lem2018-9-25-1} and (\ref{eqn2018-12-13-6}).
\end{proof}

Let $M_t=\{x\in M\, \vert\, f(x)<t\}$.
\begin{theorem} \label{thm101}
Let $n \ge 5$ and $(M, g, f)$ be an $n$-dimensional  vacuum static space of compact level sets 
of $f$ with ${\rm div}^2  B=0$.  If ${\rm div}^3 T=0$ and $f$ has its minimum (or maximum) in $M$, then $(M, g)$ is Bach-flat. 
\end{theorem}
\begin{proof}
 From Lemma~\ref{lem3},
\be \frac 12 \langle {\rm div} \, T, z\rangle =- \frac {n-2}{n-4}\, {\rm div} B(\nabla f).\label{eqs5}\ee
Since
$$ \int_{M_t} f \, {\rm div}^2 B= \int_{\Gamma_t}  \frac f{|\nabla f|} {\rm div} B(\nabla f) 
- \int_{M_t} {\rm div} B(\nabla f),$$
we have
\be 
\int_{M_t} {\rm div} B(\nabla f)= t\,  \int_{\Gamma_t}\frac 1{|\nabla f|} {\rm div} B(\nabla f)
= t\, \int_{M_t} {\rm div}^2 B=0.\label{eq2}
\ee
Therefore, from  (\ref{eqs5}) and (\ref{eq2}),  we have
$$
\frac 12 \int_{M_t} \langle {\rm div}\,  T, z\rangle = -\frac {n-2}{n-4}\int_{M_t} {\rm div} B(\nabla f)=0.
$$
On the other hand, it is easy to see that 
$$ 
\int_{\Gamma_t}\frac 1{|\nabla f|}\langle i_{\nabla f}T, z\rangle
= \int_{M_t} {\rm div} (T(\cdot, E_j,E_k)z_{jk}) =  \int_{M_t}\langle {\rm div} \, T, z\rangle 
+ \frac 12 \int_{M_t} \langle T, C\rangle.
$$
Note that, from Proposition~\ref{lem2018-9-25-1} and  (\ref{eqn2018-12-13-6}), we have
$$ 
\frac 12 \int_{M_t} \langle T, C\rangle = \frac {n-2}{n-4}\int_{M_t} {\rm div}\,  B(\nabla f)=0.
$$
Since, from Lemma~\ref{lem2018-9-25-3},
$$\langle i_{\nabla f}T, z\rangle = \frac {n-2}2\, |T|^2,$$
we have
$$ 0=\int_{M_t} \langle {\rm div}\,  T, z\rangle 
= \int_{\Gamma_t}\frac 1{|\nabla f|}\langle i_{\nabla f}T, z\rangle
=  \frac {n-2}2\int_{\Gamma_t} \frac {|T|^2}{|\nabla f|}.
$$
As a result, we may conclude that $T=0$ on $\Gamma_t$ for all $t$, which implies that
$T=0$ on all of $M$. 
Therefore, the proof follows from Proposition~\ref{prop2018-1-20-11}.
\end{proof}

\begin{theorem} \label{main2-100-100}
Let $(M, g, f)$ be a four-dimensional non-trivial complete vacuum static space with compact level sets of $f$. If
$$
 \frac 12 |C|^2= - \langle {\rm div} \, C, z\rangle, $$ 
${\rm div}^3\, T=0$,  and $f$ has its minimum (or maximum)  in $M$, then $(M, g)$ is Bach-flat.\end{theorem}
\begin{proof}
Recall that,  for $n=4$, we always have ${\rm div} B=0$.
Thus, similarly to Lemma~\ref{lem2018-2-10-3}, we have
$$ 
(n-1)\, {\rm div}^2 T (X)= -\mathring{\mathcal W}z(\nabla f, X) +\frac 12 f \langle i_XC, z\rangle,
$$
and
$$ {\rm div}^3 T = \langle i_{\nabla f}C, z\rangle +\langle {\rm div} \, T, z\rangle.$$
Thus, if ${\rm div}^3 T=0$, 
$$
\langle {\rm div} \, T, z\rangle =  -\langle i_{\nabla f}C, z\rangle.
$$
Next, as in the proof of Lemma~\ref{lem2}, we have
$$ 
\int_{\Gamma_t}\frac 1{|\nabla f|} \langle i_{\nabla f}C, z\rangle  =0
$$
when $n=4$ under the condition  $\frac 12 |C|^2 = - \langle {\rm div} C, z\rangle$.
Therefore,  from the co-area formula, we have
$$
\int_{M_t}\langle i_{\n f}C, z\rangle = \int_{M_t}|\n f|\langle i_N C, z\rangle
=\int_a^t\left(\int_{\Gamma_s} \langle i_N C, z\rangle \, d\sigma\right )\, ds =0,
$$
where $a = \min_M  f$. 
Finally,  as in the proof of Theorem~\ref{thm101},  we have
$$
\int_{\Gamma_t}\frac 1{|\nabla f|} \langle i_{\nabla f}T, z\rangle 
=\int_{M_t} \langle {\rm div} \, T, z\rangle +\frac 12\int_{M_t}\langle i_{\nabla f}C, z\rangle
= - \frac 12\int_{M_t}\langle i_{\nabla f}C, z\rangle =0.
$$
Since 
$$
\langle i_{\nabla f}T, z\rangle = \frac{n-2}{2} |T|^2,
$$
we may conclude that $T$ vanishes identically on $M$.

\end{proof}
By combining Theorem~\ref{thm101} and \ref{main2-100-100}, Theorem~\ref{thm101b} follows.

\end{document}